\documentclass[12pt]{amsart}
\usepackage{amscd,amsmath,amsthm,amssymb,graphics}
\usepackage{pstcol,pst-plot,pst-3d}
\usepackage{lmodern,pst-node}
\usepackage{multicol}
\usepackage{epic,eepic}
\usepackage{amsfonts,amssymb,amscd,amsmath,enumerate,verbatim}
\psset{unit=0.7cm,linewidth=0.8pt,arrowsize=2.5pt 4}

\newpsstyle{fatline}{linewidth=1.5pt}
\newpsstyle{fyp}{fillstyle=solid,fillcolor=verylight}
\definecolor{verylight}{gray}{0.97}
\definecolor{light}{gray}{0.9}
\definecolor{medium}{gray}{0.85}
\definecolor{dark}{gray}{0.6}

\def\frk{\frak}               

\def\Phi{{\frk n}}
\def\Phi{{\frk N}}
\def\opn#1#2{\def#1{\operatorname{#2}}} 
\opn\chara{char} \opn\length{\ell} \opn\pd{pd} \opn\rk{rk}
\opn\projdim{proj\,dim} \opn\injdim{inj\,dim} \opn\rank{rank}
\opn\depth{depth} \opn\grade{grade} \opn\height{height}
\opn\embdim{emb\,dim} \opn\codim{codim}

\opn\Tr{Tr} \opn\bigrank{big\,rank}
\opn\superheight{superheight}\opn\lcm{lcm}
\opn\trdeg{tr\,deg}
\opn\reg{reg} \opn\lreg{lreg} \opn\ini{in} \opn\lpd{lpd}
\opn\size{size}\opn\bigsize{bigsize}
\opn\cosize{cosize}\opn\bigcosize{bigcosize}
\opn\sdepth{sdepth}\opn\sreg{sreg}
\opn\link{link}\opn\fdepth{fdepth}
\opn\div{div} \opn\Div{Div} \opn\cl{cl} \opn\Cl{Cl}
\opn\Spec{Spec} \opn\Supp{Supp} \opn\supp{supp} \opn\Sing{Sing}
\opn\Ass{Ass} \opn\Min{Min}\opn\Mon{Mon} \opn\dstab{dstab} \opn\astab{astab}
\opn\Syz{Syz}
\opn\Ann{Ann} \opn\Rad{Rad} \opn\Soc{Soc}
\opn\Im{Im} \opn\Ker{Ker} \opn\Coker{Coker} \opn\Am{Am}
\opn\Hom{Hom} \opn\Tor{Tor} \opn\Ext{Ext} \opn\End{End}
\opn\Aut{Aut} \opn\id{id}

\opn\nat{nat}
\opn\pff{pf}
\opn\Pf{Pf} \opn\GL{GL} \opn\SL{SL} \opn\mod{mod} \opn\ord{ord}
\opn\Gin{Gin} \opn\Hilb{Hilb}\opn\sort{sort}
\opn\initial{init}
\opn\ende{end}
\opn\height{height}
\opn\type{type}
\opn\aff{aff} \opn\con{conv} \opn\relint{relint} \opn\st{st}
\opn\lk{lk} \opn\cn{cn} \opn\core{core} \opn\vol{vol}
\opn\link{link} \opn\star{star}\opn\lex{lex}
\opn\gr{gr}

\def\pot#1#2{#1[\kern-0.28ex[#2]\kern-0.28ex]}

\opn\dirlim{\underrightarrow{\lim}}
\opn\inivlim{\underleftarrow{\lim}}
\def\Implies{\ifmmode\Longrightarrow \else
        \unskip${}\Longrightarrow{}$\ignorespaces\fi}
\def\implies{\ifmmode\Rightarrow \else
        \unskip${}\Rightarrow{}$\ignorespaces\fi}
\def\iff{\ifmmode\Longleftrightarrow \else
        \unskip${}\Longleftrightarrow{}$\ignorespaces\fi}

\let\:=\colon
\newtheorem{Theorem}{Theorem}[section]
 \newtheorem{Lemma}[Theorem]{Lemma}
 \newtheorem{Corollary}[Theorem]{Corollary}

 \newtheorem{Example}[Theorem]{Example}
 
 \newtheorem{Definition}[Theorem]{Definition}

\let\epsilon\varepsilon
\let\kappa=\varkappa
\def\qed{\ifhmode\textqed\fi
      \ifmmode\ifinner\quad\qedsymbol\else\dispqed\fi\fi}
\def\textqed{\unskip\nobreak\penalty50
       \hskip2em\hbox{}\nobreak\hfil\qedsymbol
       \parfillskip=0pt \finalhyphendemerits=0}
\def\dispqed{\rlap{\qquad\qedsymbol}}

\opn\dis{dis}
\def\pnt{{\raise0.5mm\hbox{\large\bf.}}}

\opn\Lex{Lex}

\begin{document}

\author[Mafi and Naderi]{Amir Mafi* and Dler Naderi}
\title{Integral closure and Hilbert series of a special monomial ideal}

\address{Amir Mafi, Department of Mathematics, University Of Kurdistan, P.O. Box: 416, Sanandaj, Iran.}
\email{a\_mafi@ipm.ir}

\address{Dler Naderi, Department of Mathematics, University of Kurdistan, P.O. Box: 416, Sanandaj,
Iran.}
\email{dler.naderi65@gmail.com}

\keywords{Hilbert Series, Monomial ideal, Integral closure, Freiman ideal.\\
*Corresponding author}

\begin{abstract}
Let $R=K[x_1,\ldots, x_n]$ be the polynomial ring in $n$ variables over a field $K$ and let $M_{n,t}=(x^{e_1},\ldots, x^{e_n})$ be a monomial ideal of $R$, where $x^{e_i}=x_1^t\ldots x_{i-1}^tx_{i+1}^t\ldots x_n^t$. We study the unmixedness of its integral closure. Furthermore, we compute the Hilbert series of this ideal and we show that this ideal is Freiman.

\end{abstract}

\maketitle

\section*{Introduction}
Throughout this paper, we assume that $R=K[x_1,\ldots, x_n]$ is the polynomial ring in $n$ variables over a field $K$ with the maximal ideal $\frak{m}=(x_1,\ldots, x_n)$. Let $I$ be a monomial ideal of $R$ and suppose $\Ass(I)$ is the set of associated prime ideals of $R/I$. Brodmann \cite{B1} showed that there exists an integer $t_0$ such that $\Ass(I^t)=\Ass(I^{t_0})$ for all $t\geq t_0$. The smallest such integer $t_0$ is called the {\it index of Ass-stability} of $I$, and denoted by $\astab(I)$. Brodmann \cite{B} also showed that there exists an integer $t_0$ such that $\depth R/I^t=\depth R/I^{t_0}$ for all $t\geq t_0$. The smallest such integer $t_0$ is called the {\it index of depth stability} of $I$ and denoted by $\dstab(I)$.

Let $d, a_1,..., a_n$ be positive integers. We put $I_{(d;a_1,...,a_n)}\subset R$ be the monomial ideal generated by the monomials $u\in R$ of degree $d$ satisfying $\deg_{x_i}(u)\leq a_i$ for all $i=1,...,n$. Monomial ideals of this type are called ideals of Veronese type.
Monomial ideals of Veronese type are polymatroidal. In addition, if $a_i=1$ for all $i$, then we use $I_{d;n}$ instead of $I_{(d;1,\ldots,1)}$.
Note that the product of polymatroidal ideals is again polymatroidal (see \cite[Theorem 5.3]{CH}). In particular each power of a polymatroidal ideal is polymatroidal. In addition, every polymatroidal ideal is a normal ideal (see \cite[Theorem 3.4]{HRV}).

Katzman \cite{K} gave an explicit formula for the Hilbert series of the Veronese type ideals ( see also \cite{DH}).
Jarrah \cite{J} presented $M_{n,t}=(x^{e_1},\ldots, x^{e_n})$ a monomial ideal of $R$, where $x^{e_i}=x_1^t\ldots x_{i-1}^tx_{i+1}^t\ldots x_n^t$ and he proved that these ideals are unmixed monomial ideals of $R$ such that their integral closures have embedded primes.
 In \cite{HHZ} the equigenerated monomial ideal $I$ of $R$ is called a Freiman ideal, if $\mu(I^2)=\ell(I) \mu(I)-{\ell(I) \choose 2}$, where $\ell(I)$ is the analytic spread of $I$, that is, the dimension of $\mathcal{R}(I)/{{\frak{m}}\mathcal{R}(I)}$.

The purpose of this paper is to give a simply proof for the main theorem of \cite{J} and also we compute the Hilbert series of $M_{n,t}$.
Furthermore, we show that $M_{n,t}$ is a Freiman ideal.
For any unexplained notion or terminology, we refer the reader to \cite{HH1} or \cite{V}.

\section{Integral closure of a special monomial ideal}

We start this section by the following definition.
\begin{Definition}
Let $I$ be an ideal of $R$. Then an element $\alpha\in R$ is integral over $I$, if there is an equation
\[ {\alpha}^k+a_1{\alpha}^{k-1}+\ldots+a_{k-1}\alpha+a_k=0,\]
with $a_i\in I^i$. The set of elements $\overline{I}$ in $R$ which are integral over $I$ is the integral closure of $I$. The ideal $I$ is integrally closed, if $I=\overline{I}$, and $I$ is normal if all powers of $I$ are integrally closed. If $u\in\frak{m}^{k-1}$ or if $u\in\frak{m}\setminus\frak{m}^2$, then the ideal $(u)+\frak{m}^k$ is integrally closed (see \cite[Exercises 1.18, 1.19]{HS}).
Moreover, if $I$ is monomial ideal, then $\overline{I}$ is a monomial ideal generated by all monomials $u\in R$ for which there exists an integer $k$ such that $u^k\in I^k$.
\end{Definition}
\begin{Lemma}\label{L1}
Let $I=(u_1,\ldots, u_k)$ be a monomial ideal of $R$ and $J=(u_1^t,\ldots,u_k^t)$, where $t$ is a natural number. Then $\overline{J}=\overline{I^t}$.

\end{Lemma}

\begin{proof}
Suppose that $a\in I^t$. Then $a=u_1^{e_1}\ldots u_k^{e_k}$, where $\Sigma_{i=1}^k e_i=t$. It therefore follows $a^t=(u_1^t)^{e_1}\ldots(u_k^t)^{e_k}\in J^t$. Consider $x^t-a^t$ as an equation of integral dependence and so $a\in\overline{J}$. Thus $I^t\subseteq\overline{J}$ and hence $\overline{I^t}\subseteq\overline{J}$. For the converse it is clear $J\subseteq I^t$ and so $\overline{J}\subseteq\overline{I^t}$. Therefore $\overline{J}=\overline{I^t}$, as required.
\end{proof}

\begin{Example}Let $n=2$ and $I=(x_1^k,x_2^{k+1})$ be an ideal of $R$. Then $\overline{I}=((x_1,x_2)^{k+1},x_1^k)$.
\end{Example}

\begin{proof}
It is clear $(x_1^{k+1},x_2^{k+1})\subseteq I$ and so, by Lemma, \ref{L1} $(x_1,x_2)^{k+1}=\overline{(x_1,x_2)^{k+1}}\subseteq\overline{I}$. It therefore follows $((x_1,x_2)^{k+1},x_1^k)\subseteq\overline{I}$. Also, $I\subseteq ((x_1,x_2)^{k+1},x_1^k)$ and hance $\overline{I}\subseteq\overline{((x_1,x_2)^{k+1},x_1^k)}$.
Since $x_1^k\in (x_1,x_2)^k$, it follows that $((x_1,x_2)^{k+1},x_1^k)$ is integrally closed and so $\overline{I}\subseteq ((x_1,x_2)^{k+1},x_1^k)$. Therefore $\overline{I}=((x_1,x_2)^{k+1},x_1^k)$.
\end{proof}

\begin{Example}
Let $n=2$ and $I=(x_1,x_2^k)$. Then $I$ is normal.
\end{Example}

\begin{proof}
Since $x_1\in(x_1,x_2)\setminus(x_1,x_2)^2$, it follows $\overline{I}=\overline{((x_1,x_2)^k,x_1)}=((x_1,x_2)^k,x_1)=I$. Now, by using \cite[Theorem 1.4.10]{HS} $I$ is normal, as required.
\end{proof}

\begin{Lemma}\label{L2}
Let $I$ be a monomial ideal of $R$ such that $\Ass(I)=\{\frak{p}\}$. Then $\Ass(\overline{I})=\{\frak{p}\}$.
\end{Lemma}

\begin{proof}
Since $\sqrt{I}=\sqrt{\overline{I}}$, it follows that $\frak{p}\in\Ass(\overline{I})$. If $\frak{p}\neq\frak{q}\in\Ass(\overline{I})$, then $\frak{q}=(\overline{I}:u)$
for some monomial element $u$ of $R$. Let $x_i\in\frak{q}\setminus\frak{p}$. Then $x_iu\in\overline{I}$ and so $(x_iu)^k\in I^k$ for some $k\geq 1$. Since $I$ is a $\frak{p}-$primary ideal, it follows $u^t\in I^t$ for some positive integers $t$ and so $u\in\overline{I}$ and this is a contradiction. Therefore  $\Ass(\overline{I})=\{\frak{p}\}$, as required.
\end{proof}

\begin{Theorem}\label{T0}
Let $M_{n,t}$ and $I_{n-1;n}$ be as before. Then $\overline{M_{n,t}}=I_{n-1;n}^t$.
\end{Theorem}

\begin{proof}
As before $M_{n,t}=(x^{e_1},\ldots, x^{e_n})$, where $x^{e_i}=x_1^t\ldots x_{i-1}^tx_{i+1}^t\ldots x_n^t$. Thus, by Lemma \ref{L1}, $\overline{M_{n,t}}=\overline{I^t_{n-1,n}}$. Since Veronese type ideals are normal, it follows that $\overline{M_{n,t}}=I_{n-1;n}^t$.
\end{proof}

If you raise the $x_i$ uniformly to some power, say $x_i$ is replaced by $x_i^{t}$ everywhere in the generated of square-free Veronese type ideal $I_{n-1; n}$, then the resulting ideal $M_{n,t}$ is the image of the flat map $R\rightarrow R$ with $x_i\mapsto x_i^t$ for all $i$. Thus in this case $M_{n, t}$ will be Cohen-Macaulay, if $I_{n-1; n}$ is so ( see \cite{HTT}).
By \cite[Theorem 4.2]{HH2} the square-free Veronese type ideal $I_{n-1; n}$ is Cohen-Macaulay and so $M_{n,t}$ is Cohen-Macaulay.

\begin{Corollary}\label{C0}
Let $M_{n,t}$ be as before. Then $M_{n,t}$ is Cohen-Macaulay and $\overline{M_{n,t}}$ has embedded primes.
\end{Corollary}

\begin{proof}
By the above arguments, since $I_{n-1;n}$ is Cohen-Macaulay by \cite[Theorem 4.2]{HH2}, it follows that $M_{n,t}$ is Cohen-Macaulay. By using \cite[Lemma 5.1]{HRV}, \cite[Theorem 3.4]{V1} and Theorem \ref{T0}, $\overline{M_{n,t}}=I_{n-1;n}^t$ has embedded primes, as required.
\end{proof}

\begin{Theorem}\label{T00}
Let $M_{n,t}$ be as before. Then $\astab(M_{n,t})=\dstab(M_{n,t})=n-1$.
\end{Theorem}

\begin{proof}
Consider the flat map $R\rightarrow R$ with $x_i\mapsto x_i^t$ for all $i$. Therefore $I_{n-1;n}R=M_{n,t}$. Thus $\Ass(M_{n,t}^k)=\Ass(I^k_{n-1;n}R)$
and $\depth(R/{M_{n,t}^k})=\depth(R/{I^k_{n-1;n}R})$ for all positive integer $k$. Since by using \cite[Corollariy 5.7]{HRV} $\astab(I_{n-1;n})=\dstab(I_{n-1;n})=n-1$, it follows that $\astab(M_{n,t})=\dstab(M_{n,t})=n-1$, as required.
\end{proof}
\section{Hilbert series of a special monomial ideal}

We start this section by the following theorem.

\begin{Theorem}\label{T1}
The Hilbert function of $M_{n,t}$ is given by
\[H(i) = \sum\limits_{j = 0}^{n - 2} {{{( - 1)}^j}} \left( {\begin{array}{*{20}{c}}
n\\
j
\end{array}} \right)\left( {\begin{array}{*{20}{c}}
{i(n - 1) - ji - j + n - 1}\\
{n - 1}
\end{array}} \right).\]
\end{Theorem}

\begin{proof}

By \cite[Lemma 2.1]{DH} $H(i)$ is the number of sequences $(\beta_1, \cdots , \beta_n)$ satisfying
\linebreak
 $\sum\limits_{j = 1}^n {{\beta _j} = i(n - 1)t} $, $0 \leq \beta_{j} \leq it $ and $ \beta_{j}=t \alpha_{j}$ such that $0 \leq \alpha_{j} \leq i$. So $H(i)$ is the number of sequences $(\alpha_1, \cdots , \alpha_n)$ satisfying $\sum\limits_{j = 1}^n {{\alpha _j} = i(n - 1)} $ and $0 \leq \alpha_{j} \leq i $. In fact $H(i)$ is the number of integer solution of the equation
\[\alpha_1 + \cdots + \alpha_n = i(n-1), ~~~~ 0\leq \alpha_j \leq i, ~~~~1\leq j \leq n. \]
For $\alpha_j \geq 0$, $1 \leq j \leq n$, the number of integer solution of the equation
$\alpha_1 + \cdots + \alpha_n = i(n-1)$ is
\[\left( {\begin{array}{*{20}{c}}
{i(n - 1) + n - 1}\\
{n - 1}
\end{array}} \right).\]
For $\alpha_1 > i$ and $\alpha_j \geq 0$, $2 \leq j \leq n$, the number of integer solution of the equation
$\alpha_1 + \cdots + \alpha_n = i(n-1)$ is the same as the number of integer solution of the equation $\gamma_1 + \cdots + \gamma_n = i(n-1)-i-1$, where $\gamma_1=\alpha_1-i-1$ and  $\gamma_j =\alpha_j $ for $2 \leq j \leq n$ and equal to
\[\left( {\begin{array}{*{20}{c}}
{i(n - 1) - i - 1 + n - 1}\\
{n - 1}
\end{array}} \right). \]
For $\alpha_1, \alpha_2 > i$ and $\alpha_j \geq 0$, $3 \leq j \leq n$, the number of integer solution of the equation
$\alpha_1 + \cdots + \alpha_n = i(n-1)$ is the same as the number of integer solution of the equation $\gamma_1 + \cdots + \gamma_n = i(n-1)-2i-2$, where $\gamma_1=\alpha_1-i-1$, $\gamma_2=\alpha_2-i-1$ and  $\gamma_j =\alpha_j $ for $3 \leq j \leq n$ and equal to
\[\left( {\begin{array}{*{20}{c}}
{i(n - 1) - 2i - 2 + n - 1}\\
{n - 1}
\end{array}} \right). \]
By similar argument the number of integer solution of the equation
$\alpha_1 + \cdots + \alpha_n = i(n-1)$ such that $\alpha_{j_1}, \alpha_{j_2}, \cdots, \alpha_{j_r} > i$ for some $\{ j_1, \cdots, j_r \} \subseteq \{ 1, \cdots, n \}$ and $\alpha_j \geq 0$  for $\{ 1, \cdots, n \} \setminus \{ j_1, \cdots, j_r \}$ is equal to
\[\left( {\begin{array}{*{20}{c}}
{i(n - 1) - ji - j + n - 1}\\
{n - 1}
\end{array}} \right). \]
Now, by rule of sum we have
\[H(i) = \sum\limits_{j = 0}^{n - 2} {{{( - 1)}^j}} \left( {\begin{array}{*{20}{c}}
n\\
j
\end{array}} \right)\left( {\begin{array}{*{20}{c}}
{i(n - 1) - ji - j + n - 1}\\
{n - 1}
\end{array}} \right).\] \end{proof}

In combinatorial it is known that  for positive integer $n$ and $d$, the coefficient of $x_{1}^{n_1} \ldots x_{d}^{n_d}$ in the expansion $(x_{1}+x_{2}+ \ldots+x_{d})^n$ is
\[\frac{n!}{n_{1}! n_{2}! n_{3}! \ldots n_{d}!}\]
where each $n_{i}$ is an integer with $0 \leq n_{i} \leq n$, for all $1 \leq i \leq d$, and $n_{1} + n_{2} +\ldots+n_{d}=n$. So the coefficient of $t^i$ in the expansion $(1+t+t^2+t^3+ \ldots+t^{d-1})^{n}$ is
\[\frac{n!}{n_{1}! n_{2}! n_{3}! \ldots n_{d}!}\]
where each $n_{i}$ is an integer with $0 \leq n_{i} \leq n$, for all $1 \leq i \leq d$, $n_{1} + n_{2} +\ldots+n_{d}=n$ and $n_{2} +2n_{3}+3n_{4}+\ldots+(d-1)n_{d}=i$. Katzman \cite{K} denoted this coefficient by $A_{i}^{n,d}$. Actually
\[(1+t+t^2+t^3+ \ldots+t^{d-1})^{n}= \sum\limits_{i\geq 0}{A_{i}^{n,d}}t^{i}. \]

Katzman\cite{K} proved the following lemmas in general \cite[Theorem 2.5, Lemma 2.7 and Lemma 2.10]{K}.

\begin{Lemma}\label{l2.5}
\[(1-t)^{n-r}\sum\limits_{i\geq 0}{{i(n-j-1)+n-r-1 \choose n-r-1} t^{i}}=\sum\limits_{l\geq 0} A_{l(n-j-1)}^{n-r,n-j-1}t^{l}.\]
\end{Lemma}

\begin{Lemma}\label{l2.7}
\begin{align*}
\sum\limits_{i= 0}^{\infty}{i(n-j-1)-j+n-1 \choose n-1}t^{i}&=\sum\limits_{r= 0}^{j}(-1)^{r}{j \choose r}\sum\limits_{i= 0}^{\infty}{i(n-j-1)+n-r-1 \choose n-r-1}t^i.\\
\end{align*}
\end{Lemma}

\begin{Lemma}\label{l2.10}
\[\sum\limits_{k \geqslant 0}^{} {A_{k(n - j - 1)}^{n ,n - j - 1}}=(n-j-1)^{n-1}.\]
\end{Lemma}

Now, we are ready to prove the following result.
\begin{Theorem}
The Hilbert polynomial of $M_{n,t}$  is
\[Q(n) = \sum\limits_{j = 0}^{n - 2} {{( - 1)}^j} {n \choose j}  \sum\limits_{r = 0}^j {{( - 1)}^r}{j \choose r} {(1 - t)^r}\sum\limits_{k \geqslant 0}^{} {A_{k(n - j - 1)}^{n - 1,n - j - 1}{t^k}}. \]
\end{Theorem}
\begin{proof}
By definition of Hilbert polynomial we have
\[Q(n) = {(1 - t)^n}\sum\limits_{i = 0}^\infty  {\sum\limits_{j = 0}^{n - 2} {{{( - 1)}^j}{n \choose j}} } {i(n - 1) - ji - j + n - 1 \choose n-1}{t^i}.\]
Therefore we have
\begin{align*}
&{(1 - t)^n}\sum\limits_{i = 0}^\infty  {\sum\limits_{j = 0}^{n - 2} {{{( - 1)}^j}{n \choose j} } {i(n - 1) - ji - j + n - 1 \choose n-1}{t^i}}\\
&=\sum\limits_{j = 0}^{n - 2}{(1 - t)^n}  { {{{( - 1)}^j}{n \choose j} } \sum\limits_{i = 0}^\infty{i(n - 1) - ji - j + n - 1 \choose n-1}{t^i}}\\
&=\sum\limits_{j = 0}^{n - 2}(1-t)^{n} (-1)^j {n \choose j}  \sum\limits_{r = 0}^j {{( - 1)}^r}{j \choose r}\sum\limits_{i = 0}^\infty   {i(n-j - 1) + n -r- 1 \choose n-r-1}{t^i}\\
&=\sum\limits_{j = 0}^{n - 2}(-1)^j {n \choose j}\sum\limits_{r = 0}^j {{( - 1)}^r}{j \choose r}(1-t)^{r}(1-t)^{n-r}\sum\limits_{i = 0}^\infty   {i(n-j - 1) + n -r- 1 \choose n-r-1}{t^i}\\
&=\sum\limits_{j = 0}^{n - 2}(-1)^j {n \choose j}\sum\limits_{r = 0}^j {{( - 1)}^r}{j \choose r}(1-t)^{r}\sum\limits_{k \geqslant 0}^{} {A_{k(n - j - 1)}^{n - r,n - j - 1}{t^k}}.
\end{align*}
The second equality follows from Lemma \ref{l2.7} and the last equality follows from Lemma \ref{l2.5}.
\end{proof}

\begin{Corollary}
The Hilbert series of $M_{n,t}$ is
\[ P(n,t)=\frac{\sum\limits_{j = 0}^{n - 2} {{( - 1)}^j} {n \choose j}  \sum\limits_{r = 0}^j {{( - 1)}^r}{j \choose r} {(1 - t)^r}\sum\limits_{k \geqslant 0} {A_{k(n - j - 1)}^{n - 1,n - j - 1}{t^k}}}{(1-t)^{n}}.\]
\end{Corollary}

\begin{Corollary}
The multiplicity of $M_{n,t}$ is
\[e(M_{n,t})=\sum\limits_{j= 0}^{n-2}{(-1)^j {n \choose j} (n-j-1)^{n-1}}.\]

\end{Corollary}

If $(S, \frk{m})$ is a Noetherian
local ring and $J\subseteq I$ be a pair of ideals of $S$, then $J$ is a reduction of $I$ if
and only if $\bar{J}=\bar{I}$. In particular, $\ell(I)=\ell(J)$ ( see \cite[Theorem 3.2]{BA}).

Now, by Theorem \ref{T0} we have $\bar{M}_{n,t}$ is equal to the Veronese type ideal $I_{(t(n-1); t,\ldots,t)}$. Therefore $\ell(M_{n,t})=n$.

\begin{Theorem}
$M_{n,t}$ is a Freiman ideal for all $n$ and $t$.
\end{Theorem}
\begin{proof}
Set $I=M_{n,t}$. Then $I$ is a Freiman ideal if and only if
\[\mu(I^2)=\ell(I) \mu(I)-{\ell(I) \choose 2}.\]
$\mu(I^2)$ is the number of sequence $(\beta_{1}, \beta_{2}, \ldots, \beta_{n})$ satisfies $\beta_{1}+\beta_{2}+\ldots+\beta_{n}=2t(n-1)$ such that $\beta_{i}=0, t$ or $2t$.
So we can say, $\mu(I^2)$  is the number of sequence $(\alpha_{1}, \alpha_{2}, \ldots, \alpha_{n})$ satisfies $\alpha_{1}+\alpha_{2}+\ldots+\alpha_{n}=2(n-1)$ and $0\leq \alpha_{i}\leq 2$.

The number of integer solution of equation $\alpha_{1}+\alpha_{2}+\ldots+\alpha_{n}=2(n-1)$ such that $\alpha_{1}=0$ is one.
Therefore the number of integer solution of equation $\alpha_{1}+\alpha_{2}+\ldots+\alpha_{n}=2(n-1)$ such that $\alpha_{1}=0$ or $\alpha_{2}=0$ or $\ldots$ or $\alpha_{n}=0$ is $n={n \choose 1}$.\\
Also the number of integer solution of equation $\alpha_{1}+\alpha_{2}+\ldots+\alpha_{n}=2(n-1)$ such that $\alpha_{1}=1$ is $n-1$.
Therefore the number of integer solution of equation $\alpha_{1}+\alpha_{2}+\ldots+\alpha_{n}=2(n-1)$ such that $\alpha_{1}=1$ or $\alpha_{2}=1$ or $\ldots$ or $\alpha_{n}=1$ is
\[{n-1 \choose 1}+{n-2 \choose 1}+ \ldots+{1 \choose 1}.\]
So $\mu(I^2)={n \choose 1}+{n-1 \choose 1}+ \ldots+{1 \choose 1}={n+1 \choose 2}$.
Hence $\mu(I^2)=n \mu(I)-{n \choose 2}$ and $M_{n,t}$ is Freiman ideal for any $n$ and $t$.
\end{proof}

\end{document}